\documentclass[12pt]{amsart}

\usepackage{amsmath,amsthm,amssymb,hyperref}
\usepackage{a4wide}

\numberwithin{equation}{section}

\newtheorem{theorem}{Theorem}[section]
\newtheorem{lemma}[theorem]{Lemma}
\newtheorem{proposition}[theorem]{Proposition}
\newtheorem{corollary}[theorem]{Corollary}

\theoremstyle{remark}
\newtheorem{remark}[theorem]{Remark}
\newtheorem{remarks}[theorem]{Remarks}

\theoremstyle{remark}

\DeclareMathOperator{\Cl}{\mathrm{C}\ell}

\DeclareMathOperator{\End}{End}

\DeclareMathOperator{\id}{id}

\DeclareMathOperator{\SO}{SO}
\DeclareMathOperator{\Sp}{Sp}

\DeclareMathOperator{\spin}{\g{spin}}
\DeclareMathOperator{\Spin}{Spin}

\DeclareMathOperator{\SU}{SU}

\DeclareMathOperator{\U}{U}

\newcommand{\C}{\ensuremath{\mathbb{C}}}
\newcommand{\Ca}{\ensuremath{\mathbb{O}}}
\newcommand{\eS}{\ensuremath{\mathrm{S}}}
\newcommand{\g}[1]{\operatorname{\ensuremath{\mathfrak{#1}}}}

\newcommand{\HH}{\ensuremath{\mathrm{H}}}
\newcommand{\K}{\ensuremath{\mathbb{K}}}

\newcommand{\PP}{\ensuremath{\mathrm{P}}}
\newcommand{\R}{\ensuremath{\mathbb{R}}}
\renewcommand{\H}{\ensuremath{\mathbb{H}}}

\newcommand{\Mod}[1]{\ (\mathrm{mod}\ #1)}

\begin{document}


\title{Polar actions on Damek-Ricci spaces}

\author[A.\ Kollross]{Andreas Kollross}
\address{Institut f\"{u}r Geometrie und Topologie, Universit\"{a}t Stuttgart, Germany}
\email{kollross@mathematik.uni-stuttgart.de}

\date{\today}

\begin{abstract}
A proper isometric Lie group action on a Riemannian manifold is called polar if there exists a closed connected submanifold which meets all orbits orthogonally. In this article we study polar actions on Damek-Ricci spaces.
We prove criteria for isometric actions on Damek-Ricci spaces to be polar,
find examples and give some partial classifications of polar actions on Damek-Ricci spaces.
In particular, we show that non-trivial polar actions exist on all Damek-Ricci spaces.
\end{abstract}



\maketitle


\section{Introduction}
\label{sec:intro}


A proper isometric Lie group action on a Riemannian manifold is called \emph{polar} if there exists a connected closed submanifold, called \emph{section}, which intersects all orbits and always orthogonally.
One may think of a polar action on a Riemannian manifold as a generalization of polar coordinates on Euclidean space.
An important special case is given by actions of cohomogeneity one, i.e.~actions whose principal orbits are hypersurfaces.
In this article we study polar actions on Damek-Ricci harmonic spaces. These spaces are certain solvable Lie groups, endowed with a specific left-invariant Riemannian metric.

Polar actions are very special among general isometric Lie group actions. For example, among the countably infinitely many equivalence classes of irreducible orthogonal representations of a given semisimple compact Lie group, only a small finite number is polar~\cite{Da85}. Furthermore, it appears that for a non-trivial polar action the dimension of the group acting cannot be too small, when compared to the dimension of the space acted on, cf.~\cite{k03,k07}.
Moreover, the sections of polar actions are totally geodesic submanifolds. Thus if there exists a non-trivial polar action on a Riemannian manifold, then, generally speaking, this manifold has a high degree of symmetry and also contains many non-trivial totally geodesic submanifolds, as was remarked in~\cite{DSc}.

However, in generic Riemannian manifolds, there are neither Killing fields nor are there non-trivial totally geodesic submanifolds. It is therefore natural to study polar actions only on very special types of Riemannian manifolds.  In particular, polar actions on Riemannian symmetric spaces, which have a high degree of symmetry and contain many non-trivial totally geodesic subspaces, have been extensively studied by many authors, see e.g.~\cite{B, BDT10, BDT13, biliotti, Da85, DDK12, DK10, gk16, k02, k03, k07, k11, k20, KL12, L11, PT99, tebege, wu}. In fact, polar actions on compact Riemannian symmetric spaces are almost completely classified, see~\cite{KL12} and the references therein.

In contrast, for the symmetric spaces of the non-compact type the classification problem is still mostly open and complete classifications exist only for real~\cite{wu} and complex hyperbolic spaces~\cite{DDK12}; for example, see~\cite{DDR20} for the recently completed classification of cohomogeneity one actions on quaternionic hyperbolic spaces; the result there is surprisingly intricate.

Little is known about polar actions on non-symmetric Riemannian manifolds. Among the results of~\cite{k02} there is a classification of cohomogeneity one actions on isotropy irreducible compact homogeneous spaces (which provides many examples for such actions).

Otherwise, i.e.~apart from these cohomogeneity one examples, there are not many results concerning polar actions on Riemannian manifolds which are not (locally) symmetric.
In~\cite{DDK20} and~\cite{KP}, homogeneous Riemannian manifolds whose isotropy action is polar are studied. The result is that in the special cases which are considered (homogeneous spaces of semisimple Lie groups, generalized Heisenberg groups and Damek-Ricci spaces) the isotropy action is polar if and only if the space is locally symmetric.
In~\cite{DSc}, actions of the torus~$T^{n-1}$ on the Berger sphere~$\eS^{2n-1}$ are studied, the result is that these actions are only polar if the sphere carries the symmetric (round) metric. In particular, on homogeneous spaces which are not locally symmetric,
the only non-trivial examples of polar actions which were previously known are of cohomogeneity one.

In this article we study polar actions on Damek-Ricci spaces. It is natural to consider these spaces in the context of polar actions, as they share some properties with symmetric spaces. In fact, they were introduced by Damek and Ricci~\cite{DR} as counterexamples to the conjecture that locally harmonic manifolds are flat or rank-one locally symmetric.
Furthermore, Damek-Ricci spaces are also similar to symmetric spaces in that they contain many totally geodesic submanifolds, cf.~\cite{knp, Rou}.

Another good reason to study these spaces is that they include the non-compact symmetric spaces of rank one among them and it appears from the partial result in~\cite{k20} that it is exactly the cases missing there (i.e.\ actions given by subgroups of a maximal parabolic subgroup in the isometry group of~$\Ca\HH^2$) which correspond to the presentation of~$\Ca\HH^2$ as a Damek-Ricci space. Therefore, we may hope that progress on polar actions on Damek-Ricci spaces will also shed light on the classification problems for polar actions on~$\H\HH^n$ and ~$\Ca\HH^2$, which are still open.

We obtain new non-trivial examples of polar actions on Damek-Ricci spaces, see e.g.\ Theorem~\ref{thm:polfol}, Theorem~\ref{thm:pasl}, Corollary~\ref{cor:pfol} and Corollary~\ref{cor:psgo}.
One of our main results is the following (for the notation, see Section~\ref{sec:drsp}).

\begin{theorem}\label{thm:main}
Let $S$ be a Damek-Ricci space, $\g{s} = \g{a} \oplus \g{v} \oplus \g{z}$.
Let $\g{b} \subseteq \g{a}$ be a linear subspace and let $\g{w}$ be a linear subspace of the real vector space~$\g{v}$.
Then $\g{h} := \g{b} \oplus \g{w} \oplus \g{z}$ is a subalgebra of~$\g{s}$.
Let $H$ be the corresponding connected closed subgroup of~$S$.
Assume $Q$~acts polarly on the orthogonal complement of~$\g{w}$ in~$\g{v}$ with a section~$\ell$ such that $J_X \ell \perp \ell$ for all $X \in \g{z}$.
Then the $Q \ltimes L(H)$-action on~$S$ is polar.
\end{theorem}

This article is organized as follows. In Section~\ref{sec:drsp}, we recall preliminaries on Damek-Ricci spaces, their isometry groups and totally geodesic subspaces. In Section~\ref{sec:polact} we find criteria for the polarity of actions on Damek-Ricci spaces under the assumption that the section is a totally geodesic subgroup. We apply these criteria to obtain new examples and some partial classifications. In the last section we extend some of our results to polar actions where the sections are not totally geodesic subgroups, but hypersurfaces inside of such subgroups.

I would like to thank the anonymous referee for the careful reading of the manuscript, for pointing out some omissions and for making helpful suggestions.


\section{Damek-Ricci spaces}\label{sec:drsp}


Damek-Ricci spaces are certain solvable Lie groups, endowed with a left-invariant Riemannian metric. They are extensions of generalized Heisenberg groups, which, in turn, can be constructed from representation modules of certain Clifford algebras.  Let us briefly recall the construction, for details see~\cite{BTV95}. Let $\g{n}=\g{v}\oplus\g{z}$ be a Lie algebra, endowed with an inner product $\langle \cdot,\cdot \rangle_{\g{n}}$ such that $\langle \g{v},\g{z}\rangle_{\g{n}}=0$ and such that the Lie bracket satisfies $[\g{v},\g{v}]\subseteq\g{z}$ and $[\g{v},\g{z}]=[\g{z},\g{z}]=0$. Define a linear map $J\colon \g{z}\to \mathrm{End}(\g{v})$ by
\[
\langle J_Z U, V\rangle_{\g{n}}=\langle [U,V],Z\rangle_{\g{n}}
\]
for all $U$, $V\in\g{v}$, $Z\in\g{z}$.
The Lie algebra~$\g{n}$ is called a \emph{generalized Heisenberg algebra} if
\[
J_Z^2=-\langle Z,Z \rangle_{\g{n}} \id_{\g{v}}
\]
for all $Z\in\g{z}$.
In this case, the simply connected Lie group~$N$ with Lie algebra~$\g{n}$ is called a \emph{generalized Heisenberg group}.
Now let $\g{n}=\g{v}\oplus\g{z}$ be a generalized Heisenberg algebra and let $\g{a}$ be a one-dimensional real vector space. Let $A \in \g{a}$ be a non-zero vector. We extend the Lie bracket of~$\g{n}$ to $\g{s}:=\g{a} \oplus \g{n}$ by requiring
\[
[A,U]= \frac12U
\]
for all $U \in \g{v}$ and
\[
[A,X]= X
\]
for all $X \in \g{z}$.
We define an inner product on~$\g{s}$ by setting
\[
\langle rA+U+X , sA+V+Y \rangle := rs + \langle U , V \rangle_{\g{n}} + \langle X , Y \rangle_{\g{n}}
\]
for $r,s \in \R$, $U,V\in\g{v}$, $X,Y\in\g{z}$.
Then $\g{s}$ is a solvable Lie algebra and we define $S$ to be the simply connected Lie group with Lie algebra~$\g{s}$. The Lie group~$S$, endowed with the left-invariant Riemannian metric induced by $\langle \cdot,\cdot \rangle$, is called a \emph{Damek-Ricci space}.

The classification of Damek-Ricci spaces follows directly from the classification of Clifford modules, see e.g.~\cite{lm}, where also details on the following remarks can be found. We first recall the classification of Clifford algebras over real vector spaces with positive definite quadratic forms. They are given by Table~\ref{TCl} and the rule $\Cl_{m+8} \cong \Cl_{m} \otimes \R(16)$ for $m \ge 0$.
Here $\K(n)$ denotes the real algebra of $n \times n$-matrices with entries in~$\K = \R, \C, \H$, and $\Cl_m$ denotes the isomorphism type of a Clifford algebra over an $m$-dimensional real vector space with positive definite quadratic form.
\begin{table}[h]
\begin{tabular}{|c||c|c|c|c|c|c|c|c|} \hline
$m$ & 0 & 1 & 2 & 3 & 4 & 5 & 6 & 7 \\ \hline
$\Cl_m$ & $\R$ & $\C$ & $\H$ & $\H \oplus \H$ & $\H(2)$ & $\C(4)$ & $\R(8)$ & $\R(8) \oplus \R(8)$ \\ \hline
\end{tabular}
\hfill \\[0.5em]
\caption{Clifford algebras}\label{TCl}
\end{table}
The equivalence classes of the corresponding Clifford modules can also be read off from Table~\ref{TCl}; there is exactly one equivalence class of irreducible modules in case $m \not\equiv 3\Mod{4}$ and there are exactly two equivalence classes of irreducible modules in case $m \equiv 3\Mod{4}$. The irreducible Clifford modules are given by the natural representations of the algebra of $n \times n$-matrices with entries in~$\K$ on the real vector space~$\K^n$ in case $m \not\equiv 3\Mod{4}$ and they are given by the natural representations of the algebra of $n \times n$-matrices on the real vector space~$\K^n$, restricted to the left, resp.\ right, summand of $\K(n) \oplus \K(n)$  in case $m \equiv 3\Mod{4}$.

Accordingly, for $m \not\equiv 3\Mod{4}$  the equivalence classes of (possibly reducible)
Clifford modules of~$\Cl_m$ are given by the $k$-fold direct sums of irreducible modules, where $k$~runs through the non-negative integers; and for $m \equiv 3\Mod{4}$  the equivalence classes of (possibly reducible) Clifford modules of~$\Cl_m$ correspond to pairs~$(k_+,k_-)$ of  non-negative integers.

The generalized Heisenberg algebras are given as $\g{n} = \g{v} \oplus \g{z}$, where $\g{v}$ is a module of the Clifford algebra~$\Cl(\g{z})$, in such a fashion that the representation of the Clifford algebra~$\Cl(\g{z})$ is defined by the linear map~$Z \mapsto J_Z$, $\g{z} \to \End(\g{v})$.
The corresponding generalized Heisenberg groups~$N$ are denoted by $N(m,k)$  for $m \not\equiv 3\Mod{4}$ and by $N(m,k_+,k_-)$ for $m \equiv 3\Mod{4}$, and accordingly,  the corresponding Damek-Ricci spaces~$S$ are denoted by $S(m,k)$  for $m \not\equiv 3\Mod{4}$ and by $S(m,k_+,k_-)$ for $m \equiv 3\Mod{4}$.
The spaces $S(m,k_+,k_-)$ and $S(m,k_-,k_+)$ are isometrically isomorphic.
Let us remark that Damek-Ricci spaces with~$\g{v}\neq0$ are non-symmetric, except for the following isometries.
\begin{align*}
\R\HH^{k+1} \cong S(0,k), \;
\C\HH^{k+1} \cong S(1,k), \;
\H\HH^{k+1} \cong S(3,k), \;
\Ca\HH^2 \cong S(7,1,0) \cong S(7,0,1).
\end{align*}
Note that the spaces~$S(k,0)$, where $\g{v}=0$, are homothetic to $\R\HH^{k+1}$; there is an isomorphism $S(0,k) \to S(k,0)$, which is an isometry up to scaling.

\begin{remark}\label{rem:iso}
Let $S$ be a non-symmetric Damek-Ricci space and let $N$ be the corresponding generalized Heisenberg group, $\g{s} = \g{a} \oplus \g{n}$.
Let $A(N)$ denote the automorphisms of~$N$ whose differential at the identity element~$e$ is a linear isometry of~$(\g{n},\langle \cdot,\cdot \rangle_{\g{n}})$; these extend to automorphisms of~$S$ whose differential at the identity element~$e$ acts trivially on~$\g{a}$.
The isometry group of~$S$ is the semidirect product $A(N) \ltimes L(S)$, where for any subgroup~$H$ of~$S$ we denote by~$L(H)$ the subgroup of the isometry group of~$S$ given by left translations by elements of~$H$.
The groups $A(N)$ were determined in~\cite{riehm}.
In case $m\not\equiv 3
\Mod 4$ the connected component $A(N)_{0}$ of the group $A(N)$ is
isomorphic to
\begin{equation*}
A(N(m,k))_0  =  \left\{
\begin{alignedat}{2}
&\Spin(m)\cdot\SO(k), \quad && \hbox{~if~} m \equiv 0,6 \Mod{8}, \\
&\Spin(m)\cdot\U(k), && \hbox{~if~} m \equiv 1,5 \Mod{8}, \\
&\Spin(m)\cdot\Sp(k), && \hbox{~if~} m \equiv 2,4 \Mod{8}. \\
\end{alignedat}
\right.
\end{equation*}
In case $m\equiv 3 \Mod 4$ the connected component $A(N)_{0}$ of
the group $A(N)$ is isomorphic to
\begin{equation*}
A(N(m,k_+,k_-))_0  =  \left\{
\begin{alignedat}{2}
&\Spin(m)\cdot[\Sp(k_{+})\times\Sp(k_{-})], \quad && \hbox{~if~} m\equiv 3 \Mod{8}, \\
&\Spin(m)\cdot[\SO(k_{+})\times\SO(k_{-})], && \hbox{~if~} m\equiv 7 \Mod{8}. \\
\end{alignedat}
\right.
\end{equation*}
The actions of the groups above on~$\g{s}$ are as follows. The action is trivial on~$\g{a}$; it is equivalent to the standard~$\SO(m)$-representation on~$\g{z}$. In case $m \not \equiv 3 \Mod{4}$, the $A(N)_0$-representation on~$\g{v}$ is equivalent to $\Delta \otimes_\K \K^k$, where $\Delta$ denotes
the spin-representation of~$\Spin(m)$ on~$\K^n$, $n$ is such that $\Cl_m \cong \K(n)$ and where $\SO(k)$, $\U(k)$, $\Sp(k)$ acts by its standard representation on $\K^k = \R^k$, $\C^k$, resp.\ $\H^k$.
In case $m \equiv 3 \Mod{4}$, the $A(N)_0$-representation on~$\g{v}$ is equivalent to $(\Delta^+ \otimes_\K \K^{k_+}) \oplus (\Delta^- \otimes_\K \K^{k_-})$, where $\Delta^+$, $\Delta^-$ denote
the two half-spin-representations of~$\Spin(m)$ on~$\K^n$, $n$ is such that $\Cl_m \cong \K(n) \oplus \K(n)$ and where $\Sp(k_{\pm})$, $\SO(k_{\pm})$ act by their standard representations on $\K^{k_{\pm}} = \H^{k_{\pm}}$,  resp.\ $\K^{k_{\pm}} = \R^{k_{\pm}}$.
In case the Damek-Ricci space is a symmetric space, the isometry group of~$S$ contains the group
$A(N) \ltimes L(S)$ as a proper subgroup.
\end{remark}

A submanifold $M$ of a Riemannian manifold~$N$ is called a \emph{totally geodesic submanifold} if every geodesic of~$M$ is also a geodesic of~$N$. In~\cite{knp}, a characterization of totally geodesic submanifolds of Damek-Ricci spaces was obtained. According to their result, all complete totally geodesic submanifolds of Damek-Ricci spaces are isometric to rank-one symmetric spaces of negative curvature or are of the following type.

\begin{theorem}[\cite{Rou}]\label{thm:rou}
Let $S$ be a Damek-Ricci space.
A connected Lie subgroup~$\Sigma$ of~$S$ containing~$\exp(\g{a})$ is a totally geodesic submanifold if and only if $T_e\Sigma = \g{a} \oplus \g{v}' \oplus \g{z}'$, where $\g{v}' \subseteq \g{v}$, $\g{z}' \subseteq \g{z}$ and 
$J_{\g{z}'}\g{v}' \subseteq \g{v}'$.
\end{theorem}

In this article, we will mainly study the special case of polar actions on Damek-Ricci spaces where the section of the polar action is a totally geodesic submanifold~$\Sigma$ as described in Theorem~\ref{thm:rou}.


\section{Polar actions}\label{sec:polact}


A proper isometric action of a Lie group~$G$ on a Riemannian manifold~$M$ is called \emph{polar} if there exists a \emph{section}, i.e.\ a connected closed embedded submanifold~$\Sigma$ of~$M$ such that $G \cdot \Sigma = M$ and $T_p(G \cdot p) \perp T_p\Sigma$ for all~$p \in \Sigma$. The section of a polar action is not unique; indeed, if $\Sigma$ is a section, then $g \cdot \Sigma$ is also a section for all $g \in G$, and conversely, all sections of a polar action are given in this way. Sections of polar actions are always totally geodesic submanifolds~\cite[Theorem~3.2]{PaTe}.
We will denote the Levi-Civita connection of a Riemannian manifold by $\nabla$ and we will denote the normal space of a submanifold~$X$ at the point~$p$ by~$N_pX$.

The following lemma is a basic observation in connection with polar actions.
\begin{lemma}[\cite{DK10}]\label{lm:orkf}
Let $M$ be a Riemannian manifold and let $\Sigma$ be a connected totally
geodesic submanifold of~$M$. Let $p\in\Sigma$ and let $\eta$ be a Killing vector
field. Then $\eta(q) \in N_q\Sigma$ holds for all $q \in \Sigma$ if and only if
$\eta(p) \in N_p\Sigma$ and $\nabla_v \eta \in N_p \Sigma$ for all $v \in T_p\Sigma$.
\end{lemma}

Let us prove the following criterion for polarity.
\begin{proposition}\label{prop:pasl}
Assume a Lie group $H$ acts properly and isometrically on a connected complete Riemannian manifold~$M$ and let $p \in M$.
Let $\Sigma$ be a connected totally geodesic submanifold of~$M$ containing~$p$.
Let $H_p = \{ h \in H \mid h \cdot p = p \}$ be the isotropy subgroup of~$H$ at the point~$p$ and let $\g{h}_p$ be its Lie algebra. Let $\g{m}$ be some vector space complement of~$\g{h}_p$ in~$\g{h}$.
Then the $H$-action on~$M$ is polar with section~$\Sigma$ if and only if both the following conditions hold.
\begin{enumerate}
  \item The slice representation of~$H_p$ on the normal space~$N_p (H \cdot p)$ is polar with section~$T_p\Sigma$.
  \item We have $\langle \nabla_v \eta, w \rangle = 0$ for all Killing vector fields~$\eta$ on~$M$ induced by elements of~$\g{m}$ and all $v,w \in T_p\Sigma$.
\end{enumerate}
\end{proposition}

\begin{proof}
First assume $H$, $M$, $p$ and $\Sigma$ are as described in the hypotheses of the proposition and conditions~(i) and~(ii) hold.
Then it follows that $T_p\Sigma \subseteq N_p(H \cdot p)$ and the slice representation of~$H_p$ on~$N_p(H \cdot p)$ is such that $T_p\Sigma$ meets all $H_p$-orbits. In order to apply~\cite[Corollary~6]{DK10}, we need to show that $\nabla_v\eta \in N_p\Sigma$ for all $v \in T_p\Sigma$ and all Killing vector fields $\eta$ induced by the $H$-action.
Since we assume~(ii) holds, it suffices to check this condition for Killing vector fields induced by elements of~$\g{h}_p$.
It follows from the proof of~\cite[Theorem~7]{DK10} that the condition holds for such Killing fields if the slice representation is polar. We have shown that the $H$-action on~$M$ is polar with section~$\Sigma$.

Conversely, it is a well-known fact~\cite[Theorem~4.6]{PaTe} for polar actions that (i)~holds and it follows from Lemma~\ref{lm:orkf} that (ii)~holds since the $H$-orbits meet~$\Sigma$ orthogonally.
\end{proof}

In order to apply the above criterion in the context of Damek-Ricci spaces, we prove the following.
\begin{lemma}\label{lm:DRpc}
Assume $\Sigma$ is a totally geodesic subgroup of a Damek-Ricci space~$S$ such that $T_e\Sigma = \g{a} \oplus \g{v}' \oplus \g{z}'$, where $\g{v}' \subseteq \g{v}$, $\g{z}' \subseteq \g{z}$.
Let $H$ be a closed subgroup of~$S$ whose Lie algebra~$\g{h}$ is orthogonal to~$T_e\Sigma$.
Let $\eta$ be a Killing vector field induced by the action of~$H$ on~$S$ by left translations.
Then the following are equivalent:
\begin{enumerate}
  \item $\langle \nabla_v \eta, w \rangle = 0$ for all $v,w \in T_e\Sigma$.
  \item $J_X\g{v}' \perp \g{v}'$, where $\eta(e)=X+U$, $U \in \g{v}$, $X \in \g{z}$.
\end{enumerate}
\end{lemma}

\begin{proof}
We will compute an expression for $\langle \nabla_v \eta, w \rangle$ for arbitrary $v,w \in T_e\Sigma$, where~$\eta$ is a Killing vector field on~$S$ induced by the $H$-action.  (Note that such a Killing field is right-invariant, but in general not left-invariant, thus the formula for the Levi-Civita connection of Damek-Ricci spaces in~\cite[4.1.6]{BTV95} cannot be applied directly.)
We will use the following equality, which  holds for
arbitrary Killing vector fields $\xi$, $\eta$, $\zeta$ on a Riemannian manifold~$(M,g)$:
\[
2g( \nabla_{\xi}\eta,\zeta ) =
g( [\xi,\eta],\zeta )+
g( [\xi,\zeta],\eta )+
g( \xi,[\eta,\zeta] ).
\]
Note that for right-invariant vector fields $\xi,\eta$ on a Lie group we have $[\xi,\eta](e)=-[\xi(e),\eta(e)]$.
Let $V,W \in \g{v}'$, $Y,Z \in \g{z}'$  and let $s,t \in \R$. Let $\xi$ be the right-invariant extension of $V + Y + sA \in \g{s}$ and let $\zeta$ be the right-invariant extension of $W + Z + tA \in \g{s}$. Let $g( \cdot, \cdot )$ be the left-invariant Riemannian metric on~$\g{s}$ which is induced by the inner product~$\langle \cdot,\cdot\rangle$ on~$\g{s}$. Then $\xi$ and~$\zeta$ are Killing vector fields on~$S$. We compute:
\begin{align*}
2g( \nabla_{\xi}\eta,\zeta )&(e) =
g( [\xi,\eta],\zeta )(e)+g( [\xi,\zeta],\eta )(e)+g( \xi,[\eta,\zeta] )(e) = \\
=&-\left\langle [V,U] + \textstyle\frac{s}2U+sX,W + Z + tA \right\rangle\\
&-\left\langle [V,W] -\textstyle\frac{t}2V -tY +\textstyle\frac{s}2W +sZ,U + X \right\rangle \\
&-\left\langle V + Y + sA,[U,W] -\textstyle\frac{t}2U -tX \right\rangle = \\
=&-\langle[V,U],Z\rangle-\langle[V,W],X\rangle-\langle Y,[U,W] \rangle = \\
= &-\langle J_Z V, U \rangle
-\langle J_X V, W \rangle
+\langle J_Y W, U \rangle = -\langle J_X V, W \rangle,
\end{align*}
where we have used that $J_ZV, J_YW \in \g{v}' \perp U$ by Theorem~\ref{thm:rou} in the last step.
This shows the equivalence in the assertion of the lemma.
\end{proof}

We can now prove the following characterization of polar foliations on Damek-Ricci spaces~$S$ given by the right cosets of a subgroup~$H$ of~$S$ such that the section is as described in Theorem~\ref{thm:rou}.

\begin{theorem}\label{thm:polfol}
Let $S$ be a Damek-Ricci space with Lie algebra~$\g{s} = \g{a} \oplus \g{v} \oplus \g{z}$, where $\g{v}\neq0$.
Assume $\Sigma$ is a connected totally geodesic proper subgroup of~$S$ such that $T_e\Sigma = \g{a} \oplus \g{v}' \oplus \g{z}'$, where $\g{v}' \subseteq \g{v}$, $\g{z}' \subseteq \g{z}$.
Then there exists a closed connected subgroup~$H \subseteq S$ which acts polarly on~$S$ by left translations such that $\Sigma$~is a section if and only if $\g{z}'=0$ and $J_{\g{z}}\g{v}' \perp \g{v}'$.
\end{theorem}

\begin{proof}
Assume $\Sigma$ is as described in the hypothesis of the theorem and that there is a closed connected subgroup~$H \subseteq S$ acting polarly on~$S$ by left translations such that $\Sigma$~is a section.
Since the orbits of the $H$-action are right cosets of elements in~$S$, it follows that they are all of the same dimension, in particular, it follows that $\dim(H) + \dim(\Sigma) = \dim(S)$ and that all slice representations are trivial.

Since $H$ and $\Sigma$ intersect perpendicularly in the identity element~$e$ of~$S$, it follows that $\g{h}$ is the orthogonal complement of $T_e\Sigma$. In particular, it is of the form $\g{h} = \g{v}'' \oplus \g{z}''$ where $\g{v}'' \subset \g{v}$, $\g{z}'' \subseteq \g{z}$.

The condition that $\g{h}$ is a subalgebra of~$\g{s}$ is then equivalent to $[\g{v}'',\g{v}''] \subseteq \g{z}''$.
Let $T,U \in \g{v}''$ and let $Z \in \g{z}'$.
Then it follows that $0 = \langle J_Z T, U \rangle = \langle [T,U], Z \rangle$.
By Theorem~\ref{thm:rou} the maps $J_Z$, $Z \in \g{z}'$, leave $\g{v}'$ invariant, thus they also leave $\g{v}''$ invariant and hence induce a linear automorphism of~$\g{v}''$ if $Z\neq0$. It follows that $\g{h}$ is a subalgebra of~$\g{s}$ if and only if $\g{v}''=0$ or $\g{z}'=0$.

It follows from Proposition~\ref{prop:pasl} and Lemma~\ref{lm:DRpc} that the $H$-action on~$S$ is polar if and only if $J_XV \perp W$ for all $V,W \in \g{v}'$ and all $X \in \g{z}''$.
Assume $\g{v}''=0$. The map $J_X$ is a linear automorphism of~$\g{v}'=\g{v}$ for any non-zero $X \in \g{z}''$; thus it follows that $\g{v}'=\g{v}=0$ or $\g{z}''=0$;
in both cases we have arrived at a contradiction, since we assume that $\g{v}\neq0$ and that $\Sigma$ is a proper subgroup of~$S$.

Hence we have shown that $\g{z}'=0$, or, equivalently, $\g{z}''=\g{z}$. Now it follows from Proposition~\ref{prop:pasl} and Lemma~\ref{lm:DRpc} that $J_XV \perp W$ for all $V,W \in \g{v}'$ and all $X \in \g{z}$.

Conversely, by Proposition~\ref{prop:pasl} and Lemma~\ref{lm:DRpc} the $H$-action on~$S$ is polar if $\g{z}'=0$ and $J_{\g{z}}\g{v}' \perp \g{v}'$.
\end{proof}

\begin{remarks}
The condition $\g{z}'=0$ means that the sections of the polar actions described in~Theorem~\ref{thm:polfol} are isometric to~$\R\HH^d$, where $d = \dim(\g{v}')+1$.
In the special case where $\g{v}'=0$, we have the action of the generalized Heisenberg group~$N$ on~$S$, which is obviously of cohomogeneity one; it generalizes the well-known codimension one foliation of hyperbolic space by horospheres.
In the special case where $\g{z}=0$, i.e.\ $S\cong \R\HH^{k+1}$, $k=\dim(\g{v})$, the condition $J_{\g{z}}\g{v}' \perp \g{v}'$ becomes trivial and the $H$-action is polar for any arbitrary linear subspace~$\g{v}'$ of~$\g{v}$, cf.~\cite{wu}. For specific examples, see Corollary~\ref{cor:pfol}.
\end{remarks}

We will now study isometric actions of groups of the form $A(N)_0 \ltimes L(Z)$ on the Damek-Ricci space~$S$, where $Z$ is the connected subgroup of~$S$ with Lie algebra~$\g{z}$. We will classify all polar actions among these. A necessary condition is that the slice representation at the identity element is polar, see~\cite[Theorem~4.6]{PaTe}. This slice representation is equivalent to the action of~$A(N)_0$ on~$\g{v}$ plus the one-dimensional trivial representation on~$\g{a}$. Hence it is polar if and only if the $A(N)_0$-action on~$\g{v}$ is. The Damek-Ricci spaces where this is the case are given by the following lemma.
\begin{lemma}\label{lm:pavp}
Let $N$ be a generalized Heisenberg group with Lie algebra $\g{n} = \g{v} \oplus \g{z}$, where $\g{v}\neq0$.
Then the action of~$A(N)_0$ on~$\g{v}$ is polar if and only if $N$ is isomorphic to one of the following:
$N(0,k), k\ge1$; 
$N(1,k), k\ge1$; 
$N(2,k), k\ge1$; 
$N(3,k_+,k_-), k_\pm\ge0, (k_+,k_-)\neq(0,0)$; 
$N(4,1)$; 
$N(5,1)$; 
$N(6,1)$; 
$N(7,1,0)$; $N(7,0,1)$; 
$N(7,2,0)$; $N(7,0,2)$; 
$N(7,3,0)$; $N(7,0,3)$; 
$N(8,1)$. 
\end{lemma}

\begin{proof}
Follows from~\cite{B,Da85}. It is clear \emph{a priori} from the results of~\cite{Da85} that we only have to check the cases where $m=\dim(\g{z}) \le 16$, $m\neq11,13,14,15$.
Let us list the $A(N)_0$-actions on~$\g{v}$ for the remaining cases:
\begin{itemize}
  \item If $0 \le m \le 3$ we have the following actions:
      $\SO(k)\curvearrowright\R^k$,
      $\U(k)\curvearrowright\C^k$,
      $\U(1)\cdot\Sp(k)\curvearrowright\H^k$,
      $\Sp(1) \cdot [\Sp(k_+)\times\Sp(k_-)]\curvearrowright\H^{k_+}\oplus\H^{k_-}$, which are all polar.

  \item If $m \in \{4,5,6,8\}$, we have the following actions: $[\Sp(1)\cdot\Sp(1)]\cdot\Sp(k)\curvearrowright\H^k\oplus\H^k$, $\Sp(2)\cdot\U(k)\curvearrowright\C^4\otimes_\C\C^k$, $\SU(4)\cdot\SO(k)\curvearrowright\R^8\otimes_\R\R^k$, $\Spin(8)\cdot\SO(k)\curvearrowright\R^{16} \otimes_\R \R^k$, which are polar if and only if $k=1$.

  \item If $m=7$ we have the action $\Spin(7)\cdot[\SO(k_+)\times\SO(k_-)]\curvearrowright\R^8\otimes_\R[\R^{k_+}\oplus\R^{k_-}]$, which is polar if and only if $(k_+,k_-) \in \{ (1,0), (0,1), (2,0), (0,2), (3,0), (0,3) \}$.

  \item If $m \in \{9,10,12,16\}$, we have the actions $\Spin(9)\cdot\U(k)\curvearrowright\C^{16}\otimes_\C\C^k$, $\Spin(10)\cdot\Sp(k)\curvearrowright\H^{16}\otimes_\H\H^k$, $\Spin(12)\cdot\Sp(k)\curvearrowright\H^{32}\otimes_\H\H^k$, $\Spin(16)\cdot\SO(k)\curvearrowright\R^{256}\otimes_\R\R^k$, none of which is polar for~$k \ge 1$.

\end{itemize}
Note that in Remark~\ref{rem:iso} the spin representation of~$\Spin(m)$ is a Clifford module of~$\Cl_m$. However, in other contexts, one also refers to modules of the Clifford algebra~$\Cl_{m-1}$ as the spin representation of~$\Spin(m)$. This is due to the fact that the group $\Spin(m)$ is contained in the even subalgebra~$\Cl_m^0$, which is isomorphic to~$\Cl_{m-1}$, see~\cite[Ch.~I, Thm.~3.7]{lm}. For example, the group $\Spin(9)$ has an irreducible effective (and polar) representation on~$\R^{16}$, which is a Clifford module of~$\Cl_8$, whereas the representation referred to in Remark~\ref{rem:iso} as the spin representation of~$\Spin(9)$ is defined on~$\C^{16}$, which is an irreducible Clifford module of~$\Cl_9$. However, viewed as a real representation, the $\Spin(9)$-representation on~$\C^{16}$ splits into two equivalent irreducible submodules and is therefore not polar~\cite{B}. For more details on the spin representations involved here, see~\cite{lm}.
\end{proof}

In particular, it follows from Lemma~\ref{lm:pavp} that the representation of~$A(N)_0$ on~$\g{v}$ is of cohomogeneity two exactly in the cases of~$N(3,k_+,k_-), k_\pm\ge1$; $N(4,1)$; $N(7,2,0)$; $N(7,0,2)$; $N(8,1)$; it is of cohomogeneity three exactly in the cases $N(7,3,0)$; $N(7,0,3)$; it is of cohomogeneity one or non-polar otherwise.
We can now prove our second main result.
\begin{theorem}\label{thm:pasl}
Let $S$ be a Damek-Ricci space. Then the action of~$H=A(N)_0 \ltimes L(Z)$ on~$S$ is polar if and only if  $S$ is one of the following: $S(k,0)$, $S(0,k)$, $S(1,k)$, $S(2,k)$, $S(3,k_+,k_-)$, $S(5,1)$, $S(6,1)$, $S(7,1,0)$, $S(7,0,1)$, $S(7,2,0)$, $S(7,0,2)$, $S(7,3,0)$, $S(7,0,3)$, where $k,k_+,k_-$ are non-negative integers.
\end{theorem}

\begin{proof}
Note that for~$S(k,0)$ (the case where $\g{v}=0$) the $H$-action induces the codimension-one foliation by horospheres on~$\R\HH^{k+1}$, which is polar.

In the case where $\g{v}\neq0$, we will use Proposition~\ref{prop:pasl} and Lemma~\ref{lm:DRpc} to decide if the actions of $H = A(N)_0 \ltimes L(Z)$ on~$S$ are polar.
We choose $\g{z}$ to be the space~$\g{m}$ in Proposition~\ref{prop:pasl}.
In all cases, the orbit $H \cdot e$ through the identity element is given by the subgroup~$Z$ of~$S$.
It follows that $N_e(H \cdot e) = \g{a} \oplus \g{v}$ and furthermore, the isotropy subgroup~$H_e$ of~$H$ at~$e$ is given by~$A(N)_0$. The slice representation of the $H$-action at the point~$e$ is therefore
equivalent to the action of~$A(N)_0$ on~$\g{v}$ plus the one-dimensional trivial module~$\g{a}$. Hence the slice representation is polar if and only if the action of~$A(N)_0$ on~$\g{v}$ is.
Since slice representations of polar actions are polar, it suffices to consider the actions as given by Lemma~\ref{lm:pavp}.

In all the cases considered below we will find a section~$\g{v}'$ for the polar $A(N)_0$-action on~$\g{v}$. Note that since we assume $\g{z} \subseteq \g{h}$, the condition that the linear subspace $\g{a} \oplus \g{v}'$ of~$\g{s}$ is a subalgebra is equivalent to
$0 = \langle [V_1, V_2] , Z \rangle = \langle J_Z V_1, V_2 \rangle$ for all $V_1,V_2 \in \g{v}'$ and all~$Z \in \g{z}$ (it is trivial if $\dim(\g{v}')=1$); in other words, it is equivalent to condition~(ii) in Lemma~\ref{lm:DRpc} and thus it suffices to check this condition in all cases.

Let us consider the actions on~$\g{v}$ given in Lemma~\ref{lm:pavp} which are of cohomogeneity one.
In this case the slice representation of the $H$-action on~$S$ at the point~$e$ is of cohomogeneity two.
Let $V$ be an arbitrary non-zero vector in~$\g{v}$.
Then the subspace of~$\g{s}$ spanned by~$A$ and~$V$ is a section of the slice representation.
It follows from Lemma~\ref{lm:DRpc} that condition~(ii) of Proposition~\ref{prop:pasl} holds,
since $\g{v}' = \R V$ is one-dimensional here and we have $\langle J_XV,V \rangle = \langle [V,V], X \rangle=0$ for all $X \in \g{z}$.

The remaining $A(N)_0$-actions in Lemma~\ref{lm:pavp} are of cohomogeneity two or three on~$\g{v}$. Let us treat these actions case by case.

We start with the case $N = N(3,k_+,k_-)$, where $k_\pm\ge1$.
In this case, $A(N)_0 \cong \Sp(1)\cdot[\Sp(k_{+})\times\Sp(k_{-})]$ and $\g{v}$, as an $A(N)_0$-module, is equivalent to $\H^{k_+} \oplus \H^{k_-}$, where $\Sp(k_{\pm})$ acts on~$\H^{k_\pm}$ by its standard representation and $\Sp(1)$ acts on both submodules $\H^{k_+}$ and $\H^{k_-}$ by multiplication with unit quaternions. This representation is of cohomogeneity two, the principal orbits are diffeomorphic to~$\eS^{4k_+-1} \times \eS^{4k_--1}$.
Let $V_1 \in \H^{k_+}$, $V_2 \in \H^{k_-}$ be arbitrary non-zero vectors.
Then a section of the $A(N)_0$-action on~$\g{a}\oplus\g{v}$ is spanned by the vectors~$A,V_1,V_2$.
Condition~(ii) of Proposition~\ref{prop:pasl} is satisfied,
since the maps $J_X$ leave the subspaces $\H^{k_+}$ and $\H^{k_-}$ invariant for all $X \in \g{z}$ and therefore we have $\langle J_XV,W \rangle = 0$ for all $X \in \g{z}$ and all $V,W \in \{V_1,V_2\}$.

Now consider the cases $N = N(7,2,0)$ and $N = N(7,0,2)$. It suffices to consider one of them, since the Damek-Ricci spaces $S(7,2,0)$ and $S(7,0,2)$ are isometric and the two corresponding actions are conjugate.
In contrast to the previous case, the action of~$A(N)_0 \cong  \Spin(7) \cdot \SO(2)$ on~$\g{v}$ here is irreducible.  Indeed, the space~$\g{v}$, as an $A(N)_0$-module is equivalent to~$\R^8 \otimes \R^2$, where $\Spin(7)$ acts on~$\R^8$ by its irreducible $8$-dimensional spin representation and $\SO(2)$ acts on~$\R^2$ by its standard representation. This representation is polar~\cite{Da85} and of cohomogeneity two. Indeed, it has the same orbits as the isotropy representation of the Riemannian symmetric space~$\SO(10) / (\SO(8)\times\SO(2))$. More precisely, the $A(N)_0$-action on~$\g{v}$ is equivalent to this isotropy representation restricted to the subgroup~$\Spin(7) \times \SO(2)$.
It is well known from the theory of Riemannian symmetric spaces that a section of both these polar representations is spanned by $e_1' \otimes e_1''$, $e_2' \otimes e_2''$, where $e_1', \dots, e_8'$ and $e_1'', e_2''$ are arbitrary orthonormal bases of~$\R^8$ resp.~$\R^2$.
Choose $V_1 := e_1' \otimes e_1''$ and $V_2 := e_2' \otimes e_2''$, where we identify~$\g{v}$ with $\R^8 \otimes \R^2$ by a $\Spin(7) \times \SO(2)$-equivariant linear isometry. As a $\Cl_7$-module, $\g{v}$ is equivalent to the orthogonal direct sum $\R^8 \oplus \R^8$ and thus we see that $J_X(V_1) \perp V_2$ for all $X \in \g{z}$. Hence condition~(ii) of Proposition~\ref{prop:pasl} is also satisfied in this case.

The proof in the cases $N = N(7,3,0)$ and $N = N(7,0,3)$ is analogous to that in the preceding paragraph.

We have shown in the three previous cases that the actions considered above are polar with section~$\Sigma$, where $\Sigma$ is the connected totally geodesic subgroup of~$S$ whose Lie algebra is~$\g{a}\oplus\g{v}'$.

Now consider the cases of $S(4,1)$ and~$S(8,1)$. We will show that in both cases condition~(ii) of Proposition~\ref{prop:pasl} does \emph{not} hold, thus the $A(N)_0 \ltimes L(Z)$-actions on these spaces are not polar. Our proof works for both cases simultaneously.
Let $\g{v}$ be an irreducible $\Cl_m$-module, where $m=4$ or $m=8$ and consider the $\Spin(m)$-action on~$\g{v}$. In both cases, this action is of cohomogeneity two. Let $x \in \g{v}$ be a vector which is contained in a principal $\Spin(m)$-orbit. Then there is a vector~$y \in \g{v}$ such that a section~$\g{v}'$ of the polar $\Spin(m)$-action containing $x$ is spanned by the two vectors $x,y\in\g{v}$.

Let $e_1, \dots, e_n \in \R^n$ be generators of~$\Cl_n$ such that $e_i e_j + e_j e_i = -2\delta_{ij}$ for $1 \le i,j \le n$. Then $\spin(n) \subset \Cl_n$ is spanned by the elements $e_i e_j$, where $1 \le i < j \le n$.
However, for all~$n$, the algebra~$\Cl_n$ is isomorphic to the even subalgebra~$\Cl_{n+1}^0$ of~$\Cl_{n+1}$, where the isomorphism is given as follows:
Let $f \colon \R^n \to \Cl_{n+1}^0$ be a linear map defined by
\[
f(e_i) = e_{n+1}' e_i', \hbox{~for~} i = 1, \dots, n,
\]
where $e_1', \dots, e_{n+1}'$ are generators of~$\Cl_{n+1}$ such that $e_i' e_j' + e_j' e_i' = -2\delta_{ij}$ for $1 \le i,j \le n+1$. Then $f$~extends to an algebra isomorphism~$\tilde f \colon \Cl_n \to \Cl_{n+1}^0$, see~\cite[Ch.~I, Thm.~3.7]{lm}.
In this way, a representation module of~$\Cl_n$  becomes a $\Spin(n+1)$-representation, more precisely, the inverse image of $\spin(n+1) \subset \Cl_{n+1}^0$ under~$\tilde f$ is spanned by the elements of
\[
\{ e_i e_j \mid 1 \le i < j \le n \} \cup \{ e_i \mid 1 \le i \le n \}.
\]
In the cases that we are interested here, i.e.\ $n=4$ and $n=8$, we have that $\Spin(n+1)$ acts transitively on the unit sphere in~$\g{v}$. Indeed, the spin representations of~$\Spin(5) \cong \Sp(2)$ and $\Spin(9)$ act transitively on the unit sphere in~$\H^2$ resp.~$\R^{16}$. This shows in both cases that, while the tangent space~$T_x(\Spin(n) \cdot x)$ is orthogonal to~$x$ and $y$,  the tangent space~$T_x(\Spin(n+1) \cdot x)$ is not orthogonal to~$y$. In particular, there is an element~$e_i \in \Cl_n$ such that $J_{e_i} x \not\perp y$, showing that condition~(ii) of Proposition~\ref{prop:pasl} is not satisfied.
\end{proof}

\begin{remarks}
The sections of the polar actions found in Theorem~\ref{thm:pasl} are isometric to $\R\HH^2$, $\R\HH^3$ or $\R\HH^4$, except for the case $S(k,0)$, where the action is of cohomogeneity one and the section is isometric to~$\R$.

We have found polar actions on symmetric spaces in the following cases:
The codimension-one homogeneous foliation by horospheres in case~$S(k,0)$.
The polar action with fixed point of~$\SO(k)$ on~$S(0,k) \cong \R\HH^{k+1}$.
The action on~$S(1,k)$ given in Theorem~\ref{thm:pasl} corresponds to the case~(ii) in~\cite[Theorem~A]{DDK12} with $n=k+1$, $\g{b}=0$, $\g{w}=0$, $\g{q}=\g{u}(k)$.
Polar actions on real and complex hyperbolic spaces have been classified in~\cite{wu} resp.~\cite{DDK12}.
For quaternionic hyperbolic space and the Cayley hyperbolic plane, there are no complete classifications known so far and we find the following examples:
on~$S(3,k,0)\cong\H\HH^{k+1}$ a cohomogeneity-two action of~$\Sp(1)\cdot\Sp(k) \ltimes \R^3$; on~$S(7,1,0)\cong\Ca\HH^2$ a cohomogeneity-two action of~$\Spin(7) \ltimes \R^7$.

The other examples found in Theorem~\ref{thm:pasl} are non-trivial polar actions on the Damek-Ricci spaces $S(2,k), k\ge1$; $S(3,k_+,k_-), k_\pm\ge1$; $S(5,1)$; $S(6,1)$; $S(7,2,0)$; $S(7,0,2)$; $S(7,3,0)$; $S(7,0,3)$, which are not (locally) symmetric.
\end{remarks}

Generalizing the arguments in the proof of~Theorem~\ref{thm:pasl}, we obtain the following criterion.
\begin{theorem}\label{thm:mthm}
Let $S$ be a Damek-Ricci space, $\g{s} = \g{a} \oplus \g{v} \oplus \g{z}$.
Let $\g{w}$ be a linear subspace of the real vector space~$\g{v}$.
Then $\g{h} := \g{w} \oplus \g{z}$ is a subalgebra of~$\g{s}$, let $H$ be the corresponding connected closed subgroup of~$S$.
Let $Q \subseteq A(N)_0$ be a connected closed subgroup whose action on~$\g{v}$ leaves~$\g{w}$ invariant.
Then the $Q \ltimes L(H)$-action on~$S$ is polar if and only if $Q$~acts polarly on the orthogonal complement of~$\g{w}$ in~$\g{v}$ with a section~$\ell$ such that $J_X \ell \perp \ell$ for all $X \in \g{z}$.
\end{theorem}

\begin{proof}
Let $\g{v}\ominus\g{w}$ denote the orthogonal complement of~$\g{w}$ in~$\g{v}$.
The slice representation of the $Q \ltimes L(H)$-action on~$S$ at the identity element is equivalent to the $Q$-representation on~$\g{v}\ominus\g{w}$ plus the one-dimensional trivial module given by~$\g{a}$, thus is polar if and only if $Q$-representation on~$\g{v}\ominus\g{w}$ is.
Hence it is a necessary condition that this representation is polar for the $Q \ltimes L(H)$-action on~$S$ to be polar and we may from now on assume that the $Q$-action on~$\g{v}\ominus\g{w}$ is polar.
Let $\ell$ be a section of this polar representation.
Let $\Sigma$ be the connected totally geodesic subgroup of~$S$ as described in Theorem~\ref{thm:rou}, where $\g{v}'=\ell$, $\g{z}'=0$.
Setting $\g{m}=\g{h}$, it now follows that the $Q \ltimes L(H)$-action on~$S$ is polar if and only if condition~(ii) of Proposition~\ref{prop:pasl} is satisfied. By Lemma~\ref{lm:DRpc}, this is the case if $J_X \ell \perp \ell$ for all $X \in \g{z}$.
\end{proof}

In particular, we obtain the following result on polar actions on quaternionic hyperbolic space, where the condition on the sections to be totally real is analogous to the condition in~\cite[Proposition~4.16]{tebege} for actions on~$\H\PP^n$.
\begin{corollary}\label{cor:qhn}
Let $S = S(3,k,0)$, $\g{s} = \g{a}\oplus\g{v}\oplus\g{z}$.
Then the real vector space~$\g{v}$ has a quaternionic structure given by $J_X,J_Y,J_Z \in \End(\g{v})$, where $X,Y,Z \in \g{z}$ are suitable orthonormal vectors.
Let $\g{w}$ be a linear subspace of the real vector space~$\g{v}$.
Then $\g{h} := \g{w} \oplus \g{z}$ is a subalgebra of~$\g{s}$, let $H$ be the corresponding connected closed subgroup of~$S$.
Let $Q \subseteq A(N)_0 = \Sp(k) \cdot \Sp(1)$ be a connected closed subgroup whose action on~$\g{v}$ leaves~$\g{w}$ invariant.
Then the $Q \ltimes L(H)$-action is polar on~$S \cong \H\HH^{k+1}$ if and only if
$Q$ acts polarly on~$\g{v}$ with a section~$\ell$ which is totally real in the sense that $J_X \ell, J_Y \ell, J_Z \ell \perp \ell$.
\end{corollary}

As an another consequence of Theorems~\ref{thm:polfol} or~\ref{thm:mthm}, we find a polar regular homogeneous foliation of codimension two on every Damek-Ricci space of dimension~$\ge3$.
\begin{corollary}\label{cor:pfol}
Let $S$ be a Damek-Ricci space, $\g{s} = \g{a} \oplus \g{v} \oplus \g{z}$.
Let $\g{w} \subset \g{v}$ be a linear subspace of codimension one.
Let $H$ be the connected Lie subgroup of~$S$ with Lie algebra~$\g{w} \oplus \g{z}$.
Then the action of~$L(H)$ on~$S$ is polar.
\end{corollary}

\begin{proof}
Using the notation of Theorem~\ref{thm:mthm}, let $Q$ be the trivial subgroup of~$A(N)_0$.
Let~$\ell$ be the orthogonal complement of~$\g{w}$ in~$\g{v}$. Then $\ell$ is a section for the (trivial) $Q$-action on~$\ell$ and the condition $J_X \ell \perp \ell$ for all $X \in \g{z}$ is trivially satisfied since $\ell$ is one-dimensional.
\end{proof}

By a similar argument, we also find a polar action with a singular orbit on every Damek-Ricci space with $\dim(\g{v})\ge2$. The actions given by the following corollary are of cohomogeneity two and the orbit through the identity element of~$S$ is of codimension three.
\begin{corollary}\label{cor:psgo}
Let $S$ be a Damek-Ricci space, $\g{s} = \g{a} \oplus \g{v} \oplus \g{z}$.
Let $T \subseteq A(H)_0$ be a maximal torus.
Let $\g{w} \subset \g{v}$ be a linear subspace of codimension two such that the $T$-action on~$\g{v}$ leaves~$\g{w}$ invariant and acts non-trivially on its orthogonal complement in~$\g{v}$.
Let $H$ be the connected Lie subgroup of~$S$ with Lie algebra~$\g{w} \oplus \g{z}$.
Then the action of~$T \ltimes L(H)$ on~$S$ is polar.
\end{corollary}

\begin{proof}
With the notation of Theorem~\ref{thm:mthm}, let $Q=T$. Let $\g{v}\ominus\g{w}$ be the orthogonal complement of~$\g{w}$ in~$\g{v}$.
Then $Q$~is a connected compact Lie group which acts non-trivially by an orthogonal representation, hence with cohomogeneity one, on~$\g{v}\ominus\g{w}$.
Let $\ell$ be a one-dimensional linear subspace of~$\g{v}\ominus\g{w}$.
Then $\ell$ is a section for the $Q$-action on~$\g{v}\ominus\g{w}$ and the condition $J_X \ell \perp \ell$ for all $X \in \g{z}$ is trivially satisfied since $\ell$ is one-dimensional.
\end{proof}

Note that the restriction on the dimensions for the corollaries above only exclude real hyperbolic spaces, for which non-trivial polar actions are known to exist by other constructions.
Thus we have found non-trivial polar actions on all Damek-Ricci spaces.


\section{Sections which are not subgroups}\label{sec:nsbg}


To further generalize our results, we make the following observation.
\begin{lemma}\label{lm:aact}
Let $M$ be a Riemannian manifold. Let $H_2$ be a Lie group which contains the closed subgroups~$G$ and~$H_1$ such that $H_2 = G H_1$.
Assume that
(i)~the $H_2$-action on~$M$ is proper and isometric,
(ii)~$H_1$ acts polarly on~$M$ with section~$\Sigma_1$,
(iii)~the $H_2$-action restricted to~$G$ leaves $\Sigma_1$ invariant and acts polarly on~$\Sigma_1$ with section~$\Sigma_2$.
Then $H_2$ acts polarly on~$M$ with section~$\Sigma_2$.
\end{lemma}

\begin{proof}
Let us first show that all $H_2$-orbits intersect $\Sigma_2$: Let $p \in M$, since $H_1$~acts polarly on~$M$ with section~$\Sigma_1$, there is an $h_1 \in H_1$ such that $h_1 \cdot p \in \Sigma_1$ and since $G$~acts polarly on~$\Sigma_1$ with section~$\Sigma_2$, there is a~$g \in G$ such that $g \cdot (h_1 \cdot p) \in \Sigma_2$; hence we have $(g h_1) \cdot p \in \Sigma_2$ and $gh_1 \in H_2$.

Now let $\g{g}$, $\g{h}_1$, $\g{h}_2$ be the Lie algebras of~$G, H_1, H_2$. Any Killing vector field~$\xi$ on~$M$ induced by an element of~$\g{h}_2$ can be written as a sum $\xi = \eta + \zeta$,
where $\eta,\zeta$ are Killing vector fields on~$M$ such that $\eta$~is induced by an element of~$\g{g}$ and $\zeta$~is induced by an element of~$\g{h}_1$. Let $q \in \Sigma_2$. Then we have $\eta(q) \perp T_q\Sigma_2$, since $G$ acts polarly on~$\Sigma_1$ with section~$\Sigma_2$ and we have $\zeta(q) \perp T_q\Sigma_2$, since $H_1$ acts polarly on~$M$ with section~$\Sigma_1$ and $\Sigma_2 \subseteq \Sigma_1$.
This shows that $\xi(q) \perp T_q\Sigma_2$.

We have shown that $H_2$ acts polarly on~$M$ with section~$\Sigma_2$.
\end{proof}

We apply this to find another type of polar actions on Damek-Ricci spaces. For these actions, the sections are not subgroups as described in Theorem~\ref{thm:rou}, but totally geodesic hypersurfaces inside of them.
\begin{corollary}\label{cor:aact}
Let $S$ be a Damek-Ricci space, $\g{s} = \g{a} \oplus \g{v} \oplus \g{z}$.
Let $\g{w}$ be a linear subspace of the real vector space~$\g{v}$.
Then $\g{h} := \g{a} \oplus \g{w} \oplus \g{z}$ is a subalgebra of~$\g{s}$, let $H$ be the corresponding connected Lie subgroup of~$S$.
Let $Q \subseteq A(N)_0$ be a connected closed subgroup whose action on~$\g{v}$ leaves~$\g{w}$ invariant.
Assume $Q$~acts polarly on the orthogonal complement of~$\g{w}$ in~$\g{v}$ with a section~$\ell$ such that $J_X \ell \perp \ell$ for all $X \in \g{z}$.
Then the action of~$Q \ltimes L(H)$ on~$S$ is polar.
\end{corollary}

\begin{proof}
Consider the subalgebra $\g{k} := \g{w} \oplus \g{z}$ of~$\g{s}$ and let $K$ be the corresponding connected subgroup of~$S$.
It follows from Theorem~\ref{thm:mthm} that the action of~$Q \ltimes L(K)$ on~$S$ is polar, let $\Sigma_1$ be a section containing~$e$.
It follows from the proof of Theorem~\ref{thm:mthm} that $\Sigma_1$ is a connected totally geodesic subgroup of~$S$ as described in Theorem~\ref{thm:rou} with $\g{z}'=0$ and $T_e\Sigma_1 = \g{a} \oplus \g{v}'$, where $\g{v}'=\ell$.
In particular, $\Sigma_1$ is isometric to real hyperbolic space of dimension~$\dim(\g{v}')+1$.
It follows that the action of~$\exp(\R A)$ on~$S$ by left translations leaves~$\Sigma_1$ invariant. It is easy to check, e.g.~using \cite[Proposition~2.3]{DDK12}, that this action on real hyperbolic space is polar.
Hence it follows from Lemma~\ref{lm:aact} that the action  of~$Q \ltimes L(H)$ on~$S$ is polar.
\end{proof}

\begin{proof}[Proof of Theorem~\ref{thm:main}]
Follows from Theorem~\ref{thm:mthm} and Corollary~\ref{cor:aact}.
\end{proof}


\section{Conclusion}\label{sec:concl}


We have found many examples of non-trivial polar actions on Damek-Ricci spaces and have given classifications in some special cases.
However, our approach can only be applied to the case where the section is a subgroup (and actions derived from those using Corollary~\ref{cor:aact}).
In order to find all polar actions by our method, one would need a more explicit description of all totally geodesic subspaces in Damek-Ricci spaces, which of course would be an important result in its own right.


\end{document}